\documentclass[a4paper,11pt]{article}

\usepackage{amsmath}
\usepackage{amsthm}
\usepackage{verbatim}
\usepackage{enumerate}
\usepackage[shortlabels]{enumitem}

\usepackage[top=3.5cm, bottom=3.5cm, left=2.5cm, right=2.5cm]{geometry}

\usepackage[utf8]{inputenc}
\usepackage[T1]{fontenc}
\usepackage[french,english]{babel} 
\usepackage{enumerate}

\usepackage{xcolor}
\usepackage[]{pdfpages}

\usepackage{titling}
\usepackage[colorlinks=false,linkcolor=blue]{hyperref}
\usepackage{csquotes}
\usepackage[nottoc, notlof, notlot, numbib]{tocbibind}
\usepackage[style=alphabetic,giveninits,doi=false,isbn=false,url=false,eprint=false,clearlang=true,sorting=nyt,maxbibnames=99]{biblatex}
\DeclareSourcemap{
      \maps[datatype=bibtex]{
      \map[overwrite=false]{
       \step[fieldsource=note]
       \step[fieldset=addendum, origfieldval, final]
       \step[fieldset=note, null]
       }
   }
   }

\usepackage{mathtools}
\usepackage{thmtools}
\usepackage{bbm}
\usepackage{amssymb}
\usepackage{mathrsfs}
\usepackage{float}
\usepackage{stmaryrd}
\usepackage[normalem]{ulem}

\usepackage{setspace}
\usepackage{array}
\usepackage{tabularx}
\usepackage{ltablex} 

\usepackage{todonotes}
\usepackage{mathtools}
\mathtoolsset{showonlyrefs}

\DeclareUnicodeCharacter{221E}{$\infty$}
\DeclareUnicodeCharacter{00B4}{\'}
\DeclareUnicodeCharacter{03BB}{$\lambda$}
\DeclareUnicodeCharacter{0393}{$\Gamma$}
\DeclareUnicodeCharacter{211D}{$\mathbb{R}$}
\DeclareUnicodeCharacter{2212}{-}

\newcommand{\R}{\mathbb{R}}

\newcommand{\E}{\mathbb{E}}
\newcommand{\ps}{\mathcal{P}} 

\renewcommand{\d}{\mathrm{d}}
\newcommand{\F}{\mathcal{F}}
\renewcommand{\P}{\mathbb{P}} 

\newtheorem{theorem}{Theorem}[section]
\newtheorem{proposition}[theorem]{Proposition}
\newtheorem{lemma}[theorem]{Lemma}

\theoremstyle{definition}

\theoremstyle{remark}
\newtheorem{rem}[theorem]{Remark}

\usepackage{authblk}

\title{\bf Optimal rate of convergence in the vanishing viscosity for quadratic Hamilton-Jacobi equations}

\begin{document}

\author[1]{Louis-Pierre \textsc{Chaintron}}
\author[2]{Samuel \textsc{Daudin}}
\affil[1]{\small DMA, École normale supérieure, Université PSL, CNRS, 75005 Paris, France}
\affil[2]{\small Université Paris Cité, Laboratoire Jacques-Louis Lions}

\date{}

\maketitle

\begin{abstract}
The purpose of this note is to provide an optimal rate of convergence in the vanishing viscosity regime for first-order Hamilton-Jacobi equations with purely quadratic Hamiltonian. We show that for a globally Lipschitz-continuous 
terminal condition the rate is of order $O (\varepsilon \log \varepsilon)$, and we provide an example to show that this rate cannot be sharpened. 
This improves on the previously known rate of convergence $O(\sqrt{\varepsilon})$, which was widely believed to be optimal. 
Our proof combines techniques involving regularization by sup-convolution with entropy estimates for the flow of a suitable version of the adjoint linearized equation.
The key technical point is an integrated estimate of the Laplacian of the solution against this flow. 
Moreover, we exploit the semiconcavity generated by the equation.
\end{abstract}

\tableofcontents

\section{Introduction}

Let us fix $T > 0$ and an integer $d \geq 1$. 
Given Lipschitz-continuous functions $f, g : \R^d \rightarrow \R$, we are interested in the $\varepsilon \rightarrow 0$ convergence of the Hamilton-Jacobi-Bellman equation (HJB)
\begin{equation}  \label{eq:HJBeps} 
\begin{cases}
- \partial_t \varphi^\varepsilon_t + \frac{1}{2} \vert \nabla \varphi^\varepsilon_t \vert^2 - f = \frac{\varepsilon}{2} \Delta \varphi^\varepsilon_t, \\
\varphi^\varepsilon_T = g.
\end{cases}
\end{equation}
It is a standard result \cite{friedman1973cauchy,amann1978some} that \eqref{eq:HJBeps} has a unique $\mathcal{C}^{1,2}$ solution $\varphi^\varepsilon : [0,T] \times \R^d \rightarrow \R$, whose gradient is bounded by the Lipschitz constants $L_f$, $L_g$ of $f$ and $g$, uniformly in $\varepsilon$:
\begin{equation} \label{eq:GradBound}
\forall \varepsilon \in (0,1], \qquad \sup_{(t,x) \in [0,T] \times \R^d} \vert \nabla \varphi^{\varepsilon}_t (x) \vert \leq L_g + T L_f, 
\end{equation} 
see e.g. \cite{lions1982generalized,lions1984two}. 
From the fundamental works \cite{evans1980solving,crandall1983viscosity,crandall1984some,bardi1997optimal}, $\varphi^\varepsilon$ converges uniformly on compact sets towards the unique viscosity solution of \begin{equation} \label{eq:HJB0}
\begin{cases}
-\partial_t \varphi^0_t + \frac{1}{2} \vert \nabla \varphi^0_t \vert^2 - f = 0, \\
\varphi^0_T = g.
\end{cases}
\end{equation}
In this setting, $\varphi^0_t$ inherits the gradient bound \eqref{eq:GradBound} from $\varphi^\varepsilon_t$ and is locally Lispchitz in time, so that $\varphi^0$ is Lebesgue-almost everywhere (a.e.) differentiable and \eqref{eq:HJB0} is satisfied a.e. A natural question is then to understand the speed of this approximation. 
The convergence rate below seems to be due to \cite{fleming1964convergence} and has been widely believed optimal.

\begin{theorem}[Sub-optimal rate] \label{thm:subopt}
There exists $C_{\mathrm{sub}} >0$ such that
\begin{equation}
\forall \varepsilon \in (0,1], \quad \sup_{(t,x) \in [0,T] \times \R^d} \vert \varphi^\varepsilon_t (x) - \varphi^0_t (x) \vert \leq C_{\mathrm{sub}} \sqrt{\varepsilon}. 
\end{equation}  
\end{theorem}

The main result of the present article is the following optimal rate for this convergence.

\begin{theorem}[Optimal rate] \label{thm:opt}
If $f$, $g$ are globally Lipschitz and semiconcave, then there exists $C_{\mathrm{opt}} >0$ such that
\begin{equation}
\forall \varepsilon \in (0,1], \quad \sup_{(t,x) \in [0,T] \times \R^d} \vert \varphi^\varepsilon_t (x) - \varphi^0_t (x) \vert \leq - C_{\mathrm{opt}} \, \varepsilon \log \varepsilon.
\end{equation}  
Furthermore, there exists an explicit constant $C(d,T,L_f,L_g,\lambda_f ,\lambda_g) > 0 $ such that
\[ \forall (t,x) \in [0,T] \times \R^d, \quad \varphi^\varepsilon_t (x) - \varphi^0_t (x) \geq \tfrac{d}{2} \varepsilon \log \varepsilon -C(d,T,L_f,L_g,\lambda_f, \lambda_g) \varepsilon, \]
where $\lambda_f$, $\lambda_g$ are the semiconcavity constants of $f$, $g$. 
If $f \equiv 0$, we further show that the same rate holds without assuming that $g$ is semiconcave.
\end{theorem}

Interestingly, the leading order term is given by a universal dimensional constant.
Our result is reminiscent of the rate of convergence of the cost in the entropic approximation of optimal transport \cite{Carlier2023}.
The latter question is related to the small noise limit in the Schrödinger bridge problem \cite{leonard2013survey}, which can be rephrased as a stochastic control problem of planning type, the initial and terminal laws being imposed.
In contrast, our control interpretation of \eqref{eq:HJBeps} lets the terminal law free but adds a terminal cost, see Section \ref{subsec:controlsto}. 
Our result is also related to the large deviation theory, since Theorem \ref{thm:opt} gives the optimal rate of convergence for the exponential moment $-\varepsilon \log \E [ e^{-g(\sqrt{\varepsilon} B_t) / \varepsilon} ]$ of a rescaled Brownian motion (Schilder's theorem), see \cite{feng2006large} for the use of \eqref{eq:HJBeps} in large deviation theory.

The proof of Theorem \ref{thm:opt} relies on the regularising effect of a suitable stochastic flow combined with the semiconcavity properties of the value functions.
The main technical novelty is Proposition \ref{pro:Regularising}-\ref{item:LapLow}, which allows for estimating integrals of the kind $\int_{t}^T \int_{\R^d} \Delta \psi_s \d \mu_s \d s$ when $\mu_s$ is the solution of the Fokker-Planck equation $\partial_s \mu_s = \nabla \cdot [ \mu_s \nabla \psi_s] + \Delta \mu_s$.  
We provide an example in Section \ref{sec:Example}, which shows that the rate cannot be better than $O(\varepsilon \log \varepsilon)$ and that the constant cannot be lower than $(d-1)/2$. 
Our result demonstrates that the purely quadratic Hamiltonian actually enhances the speed of convergence. 
To wit, in the  simplest setting without non-linearity, the optimal rate cannot be better than $\sqrt{\varepsilon}$. 
The case of semiconcave data is handled by Theorem \ref{thm:optSC} in Section \ref{subsec:SC}. 
This assumption is removed in Section \ref{subsec:NSC} when $f \equiv 0$, where we complete the proof of Theorem \ref{thm:opt}.

\subsection{A long standing history}

The convergence of \eqref{eq:HJBeps} towards \eqref{eq:HJB0} has been studied by many authors, among which \cite{Hopf1950ThePD,oleinik1963discontinuous,varadhan1966asymptotic}.
We also refer to \cite{kruzhkov1964cauchy,kuznetsov1964higher} for generalizations to systems of conservation laws.
Indeed, for $d=1$, $\nabla \varphi^0$ corresponds -- up to time-reversal -- to the unique entropy solution of the Burgers equation, which can also be obtained by vanishing viscosity methods \cite{kruvzkov1970first,lax2005hyperbolic}.

To our knowledge, the convergence rate $O(\sqrt{\varepsilon})$ was first obtained by \cite{fleming1964convergence} using a differential game approach. 
Within the framework of viscosity solutions, \cite{crandall1983viscosity,lions1984two} established it for general Hamiltonians using the doubling variables method. 
For some types of non-convex Hamiltonians, this rate is shown to be optimal in \cite{qian2024optimal}.
\cite{bardi1997optimal} further showed that this rate deteriorates if the terminal data is only Hölder-continuous. 
When the Hamiltonian is convex, a natural approach to these questions relies on the stochastic control interpretation of $\varphi^\varepsilon_t$, see Section \ref{subsec:controlsto} below.
Some other proofs leverage the nonlinear adjoint method introduced by \cite{evans2010adjoint,tran2011adjoint}.

When the Hamiltonian is uniformly convex -- as in \eqref{eq:HJBeps} --, better rates can be expected if the solution is semiconcave.
In this setting, for a semiconcave terminal condition, a $O(\varepsilon)$ upper bound is known to hold \cite{lions1982generalized,calder2018lecture,camilli2023quantitative}, see Lemma \ref{lem:UppOpt} below -- when $g$ is no more semiconcave, we further prove a  deteriorated version of this upper bound in Proposition \ref{pro:UppNSC}.
Using the non-linear adjoint method, \cite{tran2021hamilton} obtained a $O(\varepsilon)$ rate for suitable averages of $\varphi^\varepsilon - \varphi^0$.
Working in the torus, \cite{camilli2023quantitative} obtained a $O(\varepsilon)$ rate in $L^\infty ( L^1 )$.
See also \cite{Tang-Teng} for a similar one-dimensional result for conservation laws.
The global rate $O(\varepsilon \log \varepsilon)$ had already been obtained in dimension $d=1$ for purely quadratic Hamiltonians and Lipschitz terminal data, see \cite[Proposition 4.4]{qian2024optimal}, which relies on the Cole-Hopf-Cole formula \eqref{eq:HopfC} for $\varphi^\varepsilon$ and the Hopf-Lax formula for $\varphi^0$.
The works 
\cite{droniou2006fractal,goffi2024remarks} have studied regularizing effects of \eqref{eq:HJBeps} when the usual Laplacian is replaced by a non-local fractional Laplacian.
More precisely, \cite{goffi2024remarks} computed optimal rates in this case and deduced a $O(\varepsilon)$ rate for \eqref{eq:HJBeps} in $L^\infty([0,T],L^p( \R^d))$ for every $1 < p < +\infty$. 
To the best our knowledge, Theorem \ref{thm:opt} is the first result giving the lower bound in every dimension, which has been an open question for a long time. However, we have been informed that the optimal rate of convergence for uniformly convex Hamiltonians has been recently obtained in an incoming work independent of ours \cite{CirGoffOpt}.

Since the early works of Fleming, it was realized that the rate of convergence deteriorates around the singularities of the solution to the first order Hamilton-Jacobi equation. In the special case where the terminal condition $g$ is $\mathcal{C}^{1,1}$ and convex, then the solution $\varphi_t^0$ is actually $\mathcal{C}^{1,1}$ for all $t \in [0,T]$, and a simple comparison argument shows that the rate of convergence is of order $O(\varepsilon)$. When $g$ is not convex, the solution can develop singularities. 
In this case, from \cite[Example 10.2']{fleming1971stochastic} in a similar setting, we cannot expect a global rate better than $O(\varepsilon\log\varepsilon)$. On the other hand, the solution is always semiconcave and singularities of semiconcave functions cannot be too numerous. The optimal rate $O(\varepsilon)$ can then  be achieved if we restrict ourselves to points $(t,x)$ in the \emph{strong regularity region}, which is known to be a dense open subset of $[0,T] \times \R^d$ \cite[Theorem 2]{fleming1964cauchy} and \cite[Chapters 4-5]{cannarsa2004semiconcave}. We refer to \cite{fleming1964cauchy,fleming1971stochastic} for further results and an asymptotic expansion of $\varphi^\varepsilon_t$ w.r.t. $\varepsilon$ within the strong regularity region.
These works rely on a probabilistic stochastic control approach. See \cite{fleming1986asymptotic} for an analogous result  using viscosity solution methods.

\subsection{Stochastic control interpretation} \label{subsec:controlsto}

Let us fix a filtered probability space $(\Omega,\F,(\F_t)_{0 \leq t \leq T}, \P)$. For $\varepsilon > 0$, we consider the controlled dynamics
\begin{equation} \label{eq:trajeps}
 X^{\varepsilon,\alpha}_s := x + \int_t^s \alpha_s \d s + \sqrt{\varepsilon} B_s, \quad t \leq s \leq T, 
\end{equation}
where $(B_s)_{t \leq s \leq T}$ is a $(\F_t)_{0 \leq t \leq T}$-Brownian motion.
The control $\alpha = (\alpha_s)_{t \leq s \leq T}$ is now a stochastic process, which is progressively measurable and square-integrable.
Leveraging the regularity of $\varphi^\varepsilon$ and Bellman's dynamic programming principle, it is well-known that $\varphi^\varepsilon$ corresponds to the the value function of a stochastic control problem,
\begin{equation} \label{eq:Vcontroleps}
\varphi^\varepsilon_t ( x ) = \inf_{\alpha, X^{\varepsilon,\alpha}_t = x} \E \bigg[ \int_t^T \frac{1}{2} \vert \alpha_s \vert^2 + f ( X^{\varepsilon,\alpha}_s ) \d s + g ( X^{\varepsilon,\alpha}_T ) \bigg],
\end{equation} 
see e.g. the textbook \cite{fleming2006controlled} on stochastic control. The unique optimal curve is the pathwise unique solution of the stochastic differential equation (SDE)
\begin{equation} \label{eq:opteps}
\overline{X}^\varepsilon_s = x - \int_t^s \nabla \varphi^\varepsilon_s ( \overline{X}^\varepsilon_s ) \d s + \sqrt{\varepsilon}  B_s, 
\end{equation} 
so that $-\nabla \varphi^\varepsilon_t$ provides the optimal control in feed-back form.

Sending $\varepsilon \rightarrow 0$, $\varphi^0_t (x)$ is also the value function of a deterministic control problem,
\begin{equation} \label{eq:Vcontrol0}
\varphi^0_t ( x ) = \inf_{\alpha, X^{0,\alpha}_t = x} \, \int_t^T \frac{1}{2} \vert \alpha_s \vert^2 + f ( X^{0,\alpha}_s ) \d s + g ( X^{0,\alpha}_T ),
\end{equation} 
where we here minimise over \emph{deterministic controls} $\alpha \in L^2 ( (t,T), \R^d )$.
However, this problem does not have a unique optimal solution in general, unless $f$ and
$g$ are convex functions.
This lack of uniqueness is direclty related to the lack of regularity of $\varphi^0_t$, which is not $\mathcal{C}^1$ in general.
At this stage, a direct coupling argument of the value functions \eqref{eq:Vcontroleps}-\eqref{eq:Vcontrol0} yields the suboptimal rate $O(\sqrt{\varepsilon})$.

\subsection{Comparison with mean-field control} \label{subsec:MFC}

The search for the optimal rate of convergence in the vanishing viscosity problem has sparked renewed interest in recent years, partly due to analogous questions arising in mean-field optimal control theory. Mean-field optimal control theory focuses on optimal control problems for large populations of $N \geq 1$ interacting particles. A central question in this theory is to quantitatively understand the regime $N \rightarrow +\infty$. This can be approached using PDE techniques, particularly at the level of the value functions for the different problems. In certain settings, where each particle is subject to its own idiosyncratic Brownian noise, the pre-limit problem can be viewed as a noisy perturbation of the limit problem, which is formulated over a space of probability measures. In this context, the convergence problem in mean-field control theory can be seen as an infinite-dimensional analog of the vanishing viscosity problem, where the viscosity intensity is given by $ \varepsilon = 1/N$. This connection has inspired a recent line of research applying well-established techniques from the vanishing viscosity problem to the convergence problem in mean-field control theory, as seen in \cite{cdjs2023,ddj2023,cjms2023,cecchin2025quantitativeconvergencemeanfield}. The analogy extends further, as the vanishing viscosity problem can itself be viewed as a special case of the mean-field problem when the interaction is solely through the mean position of the population, as demonstrated in \cite[Example 2.13]{ddj2023}. Consequently, the optimal rate of convergence in the vanishing viscosity regime provides a lower bound on the optimal rate in the mean-field regime. In particular, Theorem \ref{thm:opt} shows that the convergence rate in \cite[Theorem 2.7]{ddj2023} -- in the so-called \textit{regular} case where the terminal cost is Lipschitz and semi-concave with respect to sufficiently weak metrics -- could be better than $N^{-1/2}$ but not better than $N^{-1} \log(N)$. Adapting the techniques presented in this paper to the mean-field setting is an exciting yet undoubtedly challenging prospect, which we leave for future work.

\subsection{Extensions and open questions}

The present article is only concerned with establishing the optimal rate of convergence for \eqref{eq:HJBeps}. 
We restricted ourselves to a minimal framework to favor the clarity of exposition. 
Extending these results to more general settings will be the goal of future works, and we list below some appealing perspectives.

\begin{itemize}
\item A straightforward extension would replace \eqref{eq:HJBeps} by the HJB equation
\[ - \partial_t \varphi^\varepsilon_t - b_t \cdot \nabla \varphi^\varepsilon_t + \frac{1}{2} \vert \sigma^\top_t \nabla \varphi^\varepsilon_t \vert^2 - f_t = \frac{\varepsilon}{2} \mathrm{Tr}[\sigma_t \sigma^\top_t \nabla^2 \varphi^\varepsilon_t ], \]
for Lipschitz coefficients $b :[0,T] \times \R^d \rightarrow \R^d$ and $\sigma :[0,T] \times \R^{d} \rightarrow \R^{d \times d}$.
This amounts to replacing the controlled dynamics \eqref{eq:trajeps} by
\[ \d X^\alpha_t = b ( X^\alpha_t ) \d t + \sigma_t ( X^\alpha_t ) \alpha_t \d t + \sqrt{\varepsilon}  \sigma_t ( X^\alpha_t ) \d B_t. \]
If $\sigma$ is globally bounded and uniformly elliptic, we anticipate that our proofs easily adapt.
The optimal feedback control in this case is given by $- \sigma^\top_t \nabla \varphi^\varepsilon_t$, and a gradient estimate similar to \eqref{eq:GradBound} was proved in \cite{chaintron2023existence}.
\item A natural question is to identify the Hamiltonians $\mathcal{H}$ for which Theorem \ref{thm:opt} extends to 
\[ -\partial_t \varphi^\varepsilon_t (x) + \mathcal{H}(t,x,\varphi^\varepsilon_t (x),\nabla \varphi^\varepsilon_t (x)) = \frac{\varepsilon}{2} \Delta \varphi^\varepsilon_t (x), \]
in place of \eqref{eq:HJBeps}. We emphasize that our proof relies on the purely quadratic structure of the Hamiltonian  -- for instance to derive \eqref{eq:KeyLower} -- and that adapting it for other Hamiltonians satisfying the assumptions of e.g. \cite{fleming1971stochastic,fleming1986asymptotic} would require new ideas. We could also think about rescaled non-linearities with respect to $\nabla^2 \varphi^\varepsilon_t$.
Homogenization problems in the spirit of \cite{qian2024optimal} would be interesting too, as well as adding boundary conditions.

\item The fundamental obstruction to the smooth-setting $O( \varepsilon )$ rate is the lack of regularity of $\varphi^0_t$.
It would be illuminating to relate the $O( \varepsilon \log \varepsilon )$ rate (and the pre-factor $d/2$) to the singularities of $\nabla \varphi^0_t$. 
Such a study will likely require a thorough analysis of the singularities in the spirit of \cite{cannarsa2004semiconcave}. 
\item We eventually refer to Section \ref{subsec:MFC} for the promising perspective of extending our results to the mean-field control setting in the spirit of \cite{ddj2023}. 
\end{itemize}

\section{Proof of the main results}

\subsection{Semiconcave setting}  \label{subsec:SC}

Let $L_f$, $L_g$ denote the Lipschitz constants of $f$, $g$, and let $\lambda_f$, $\lambda_g$ denote their semiconcavity constants.
We set $L := L_g + TL_f $ and $\lambda := \lambda_g + T \vert \lambda_f \vert$.
From the representation formula \eqref{eq:Vcontroleps}, we infer that $\varphi^\varepsilon_t$ is $\lambda$-semiconcave for every $t \in [0,T]$, and this property extends to $\varepsilon =0$. 
This is the well-known property that HJB equations propagate semiconcavity \cite{cannarsa2004semiconcave}.
The following estimate is a known fact. 

\begin{lemma}[Upper bound] \label{lem:UppOpt}
For every $\varepsilon \in (0,1]$,
\begin{equation*} \label{eq:KnownUpper}
\forall (t,x) \in [0,T] \times \R^d, \quad \varphi^{\varepsilon}_t ( x) - \varphi^{0}_t ( x) \leq \frac{(T - t )d \lambda}{2} \varepsilon.
\end{equation*}
\end{lemma}

For the sake of completeness, we provide the classical proof below, which relies on viscosity solution theory. 
A stochastic control proof is as well possible, as we do for the lower bound.

\begin{proof}
By semiconcavity, $\Delta \varphi^\varepsilon_t \leq d\lambda$, so that $(t,x) \mapsto \varphi^\varepsilon (t,x) - \varepsilon(T -  t) d\lambda / 2$ is a viscosity sub-solution of \eqref{eq:HJB0}.
By comparison with the viscosity (super-)solution $\varphi^0$, the result follows. 
\end{proof}

The main difficulty to prove the upper bound is the lack of semiconvexity for $\varphi^0_t$.
To circumvent this, we regularise $\varphi^0_t$ using the standard sup-convolution \cite{lasry1986remark}, for $\delta >0$,
\[ \varphi^{0,\delta}_t (x) := \sup_{y \in \R^d} \varphi^{0,\delta}_t (y) - \frac{1}{2 \delta} \vert x - y \vert^2. \]
The following lemma recalls known properties of sup-convolutions, see e.g. \cite{calder2018lecture,tran2021hamilton}.

\begin{lemma}[Sup-convolution]$\label{lem:supconv} \phantom{a}$
\begin{enumerate}[label=(\roman*),ref=(\roman*)]
\item\label{item:supReg} $\varphi^{0,\delta}_t$ is $L$-Lipschitz, $\lambda$-semiconcave, and $-1/\delta$-semiconvex. In particular, $\varphi^{0,\delta}_t$ is $\mathcal{C}^{1,1}$.
\item\label{item:supAppr} $\forall (t,x) \in [0,T] \times \R^d, \quad 0 \leq \varphi^{0,\delta}_t (x) - \varphi^0_t (x) \leq 2 L \delta$.
\item $\varphi^{0,\delta}_t$ inherits the sub-solution property for \eqref{eq:HJB0}, so that for a.e. $t \in [0,T]$,
\begin{equation} \label{eq:convolSub}
\forall x \in \R^d, \quad - \partial_t \varphi^{0,\delta}_t (x) + \frac{1}{2} \vert \nabla \varphi^{0,\delta}_t (x) \vert^2 -f \leq 2 \delta L L_f,
\end{equation} 
recalling that $\varphi^{0,\delta}$ is locally Lipschitz in $t$ and $\mathcal{C}^{1,1}$ in $x$. 
\end{enumerate}
\end{lemma}

Introducing the half-sum $\psi^{\varepsilon,\delta} := [ \varphi^{0,\delta} + \varphi^{\varepsilon} ]/2$, the key idea of our proof is to compute $\varphi^\varepsilon_s - \varphi^0_s$ along the solution of the stochastic differential equation (SDE)
\begin{equation} \label{eq:OInterpSDE}
\d Y^{\varepsilon,\delta,x}_s = - \nabla \psi^{\varepsilon,\delta}_s ( Y^{\varepsilon,\delta,x}_s ) \d s + \sqrt{\varepsilon} \d B_s, \quad t \leq s \leq T,
\end{equation} 
starting at $Y^{\varepsilon,\delta,x}_{t} = x$, referring to the setting of Section \ref{subsec:controlsto}. 
Since $\nabla \psi^{\varepsilon,\delta}$ is globally Lipschitz, this equation has a pathwise unique solution.
We now look at the law $\mu^{\varepsilon,\delta,x}_s \in \ps ( \R^d )$ of $Y^{\varepsilon,\delta,x}_s$. 

\begin{proposition}[Key regularizing effect] \label{pro:Regularising}
For every $\tau \in (0,T-t]$, $\mu^{\varepsilon,\delta,x}_{t+\tau} \in \ps ( \R^d )$ has a positive density, still denoted by $\mu_{t +\tau}^{\varepsilon,\delta,x} \in L^1 ( \R^d )$, which satisfies
\begin{enumerate}[label=(\roman*),ref=(\roman*)]
\item\label{item:RegEnt} 
$ - \infty < \int_{\R^d} \log \mu^{\varepsilon,\delta,x}_{t+\tau} \d \mu^{\varepsilon,\delta,x}_{t+\tau} \leq- \frac{d}{2} \log ( 2 \pi \varepsilon \tau ) + \frac{\tau}{2 \varepsilon} L^2$,
\item\label{item:LapLow} $\frac{1}{2} \int_{t+\tau}^T \int_{\R^d} \Delta \varphi^{0,\delta}_s \d \mu^{\varepsilon,\delta, x}_s \d s \geq \frac{d}{2} \log ( 2 \pi \varepsilon \tau ) - \frac{\tau}{2 \varepsilon} L^2 - \frac{d}{2} \log (2 \pi \varepsilon ) - \frac{(T-t-\tau) d \lambda}{2} - 2 d e^{4 L^2 T^2}$.
\end{enumerate} 
\end{proposition}

The bound \ref{item:RegEnt} can be obtained through several ways. A natural approach is to leverage known Gaussian bounds for Green functions \cite{aronson1968non,sheu1991some}.
The optimal bounds for heat kernels \cite{davies1987explicit} would also work here, exploiting the gradient structure of \eqref{eq:OInterpSDE} using \cite[Proposition 3.1]{bakry1997optimal}.
The following short proof has been suggested to the first author by Daniel Lacker.

\begin{proof}
To alleviate notations, we omit the superscript $x$ in this proof.
\medskip

\ref{item:RegEnt} Let $w^{\varepsilon}_{[t,t+\varepsilon]} \in \ps ( \mathcal{C} ( [t,t+\tau] , \R^d ))$ denote the path law of $( x + \sqrt{\varepsilon} B_{s-t} )_{t \leq s \leq t + \tau}$.
The marginal law $w^{\varepsilon}_s$ at time $s$ is Gaussian centered at $x$ with variance $\varepsilon (s-t)$.
Let $\mu^{\varepsilon,\delta}_{[t,t+\varepsilon]}$ denote the path law of $(Y^{\varepsilon,\delta}_s)_{t \leq s \leq t +\tau}$.
The Girsanov transform computes the pathwise relative entropy
\[ H ( \mu^{\varepsilon,\delta}_{[t,t+\tau]} \vert w^{\varepsilon}_{[t,t+\tau]} ) = \frac{1}{2\varepsilon} \E \int_{t}^{t+\tau} \vert \nabla \psi^{\varepsilon,\delta}_s ( Y^{\varepsilon,\delta}_s) \vert^2 \d s \leq \frac{\tau}{ 2 \varepsilon} L^2. \]
From the contraction property of relative entropy \cite[Theorem D.13]{dembo2009large}, we get $H ( \mu^{\varepsilon,\delta}_{t+\tau} \vert w^{\varepsilon}_{t+\tau} ) \leq H ( \mu^{\varepsilon,\delta}_{[t,t+\tau]} \vert w^{\varepsilon}_{[t,t+\tau]} )$, so that
\[ \int_{\R^d} \log \mu^{\varepsilon,\delta}_{t+\tau} \d \mu^{\varepsilon,\delta}_{t+\tau} \leq \int_{\R^d} \log w^{\varepsilon}_{t+\tau} \d \mu^{\varepsilon,\delta}_{t+\tau} + \frac{\tau}{2 \varepsilon} L^2 \leq - \frac{d}{2} \log ( 2 \pi \varepsilon \tau )+ \frac{\tau}{2 \varepsilon} L^2, \]
using the pointwise bound for the Gaussian density on $\R^d$.
\medskip

\ref{item:LapLow} Let us first assume that $\mu^{\varepsilon,\delta}$ is in $\mathcal{C}^{1,2} ( (0,T) \times \R^d)$. Since $\mu^{\varepsilon,\delta}$ is positive and satisfies the Fokker-Planck equation
\[ \partial_s \mu^{\varepsilon,\delta}_s = \nabla \cdot \big[ \mu^{\varepsilon,\delta}_s \nabla \psi^{\varepsilon,\delta}_s + \frac{\varepsilon}{2} \nabla \mu^{\varepsilon,\delta}_s \big], \]
we get that
\begin{equation} \label{eq:Log}
\big[ \partial_s - \nabla \psi^{\varepsilon,\delta}_s \cdot \nabla + \frac{\varepsilon}{2} \Delta \big] \log \mu^{\varepsilon,\delta}_s = \Delta \psi^{\varepsilon,\delta}_s - \frac{\varepsilon}{2} \vert \nabla \log \mu^{\varepsilon,\delta}_s \vert^2.  
\end{equation} 
Ito's formula for \eqref{eq:OInterpSDE} then almost surely yields
\begin{multline} \label{ItoLog}
\log \mu^{\varepsilon,\delta}_{T} ( Y^{\varepsilon,\delta}_{T} ) - \log \mu^{\varepsilon,\delta}_{t+\tau} ( Y^{\varepsilon,\delta}_{t+\tau} ) = \int_{t+\tau}^T \Delta \psi_s^{\varepsilon,\delta} ( Y^{\varepsilon,\delta}_s ) - \frac{\varepsilon}{2} \vert \nabla \log \mu^{\varepsilon,\delta}_{s} ( Y^{\varepsilon,\delta}_{s} ) \vert^2 \, \d s \\
+ \sqrt{\varepsilon} \int_{t+\tau}^T \nabla \log \mu^{\varepsilon,\delta}_s (Y^{\varepsilon,\delta}_s) \cdot \d B_s. 
\end{multline} 
The last term is a martingale, using the finiteness of the integrated Fisher information given by e.g. \cite[Theorem 7.4.1]{bogachev2022fokker}.
We now take expectations and we use $\psi^{\varepsilon,\delta} = [ \varphi^{\varepsilon} + \varphi^{0,\delta} ]/2$ to write
\[ \frac{1}{2} \int_{t+\tau}^T \int_{\R^d} \Delta \varphi^{0,\delta}_s \d \mu^{\varepsilon,\delta}_s \d s \geq - \frac{1}{2} \int_{t+\tau}^T \int_{\R^d} \Delta \varphi^{\varepsilon}_s \d \mu^{\varepsilon,\delta}_s \d s + \int_{\R^d} \log \mu^{\varepsilon,\delta}_{T} \d \mu^{\varepsilon,\delta}_{T} - \int_{\R^d} \log \mu^{\varepsilon,\delta}_{t+\tau} \d \mu^{\varepsilon,\delta}_{t+\tau}, \]
where we got rid of the non-positive integrated Fisher information term.
We lower bound the first term on the r.h.s. using the semiconcavity of $\varphi^\varepsilon_s$, and the last term using \ref{item:RegEnt}. For the middle term, denoting the Gaussian law $\mathcal{N}(x,{\varepsilon \mathrm{Id}})$ by $\gamma^\varepsilon$,
\[ \int_{\R^d} \log \mu^{\varepsilon,\delta}_{T} \d \mu^{\varepsilon,\delta}_{T} = H ( \mu^{\varepsilon,\delta}_{T} \vert \gamma^\varepsilon ) + \int_{\R^d} \log \gamma^\varepsilon \d \mu^{\varepsilon,\delta}_{T} \geq - \frac{d}{2}\log( 2 \pi \varepsilon ) - \frac{1}{2 \varepsilon} \int_{\R^d} ( y - x)^2 \d \mu^{\varepsilon,\delta}_{T} ( y ). \]
Furthermore, using Ito's formula and Jensen's inequality 
\[ \frac{\d}{\d s} \int_{\R^d} ( y - x)^2 \d \mu^{\varepsilon,\delta}_{s} ( y ) \leq 2 L \bigg( \int_{\R^d} ( y - x)^2 \d \mu^{\varepsilon,\delta}_{s} \bigg)^{1/2} + 2 d \varepsilon, \]
so that $\int_{\R^d} ( y - x)^2 \d \mu^{\varepsilon,\delta}_{T} ( y ) \leq 4 d \varepsilon e^{4 L^2 T^2}$, using a variation on the Gronwall lemma.
Gathering terms yields the r.h.s of \ref{item:LapLow}.

In general $\mu^{\varepsilon,\delta}$ is not $\mathcal{C}^{1,2}$, but we can regularise $\psi^{\varepsilon,\delta}$ and use a standard approximation argument.
Indeed, the Green function $\mu^{\varepsilon,\delta}$ is in $L^2 ( [t+\tau,T], H^1 ( \R^d))$ from \cite[Theorem 9]{aronson1968non}, and can be approximated weakly in $L^2 ( [t+\tau,T], H^1 ( \R^d))$ in this way by smooth Green functions that pointwise converge to it \cite[Lemma 7]{aronson1968non}.
\end{proof}

Let us first estimate the error near the terminal time.

\begin{lemma}[Terminal continuity] \label{lem:TemrCont}
If $f$ and $g$ are only Lipschitz-continuous, then for every $\varepsilon \in (0,1]$ and $\eta \in [0,T]$,
\[ \sup_{(t,x) \in [T-\eta,T] \times \R^d} | \varphi_t^{\varepsilon}(x) - \varphi_t^0(x)| \leq 3 L^2 \eta +  2 L \sqrt{\varepsilon} \sqrt{\eta}. \]
\end{lemma}

\begin{proof}
We leverage the control tools from Section \ref{subsec:controlsto}. From
\eqref{eq:Vcontroleps}-\eqref{eq:opteps}, for every $(t,x) \in [0,T] \times \R^d$,
\[ \varphi_t^{\varepsilon}(x) = \E \bigg[ \frac{1}{2}\int_{t}^T | \nabla \varphi_s^{\varepsilon}( \overline{X}^{\varepsilon}_s)|^2 + f ( \overline{X}^\varepsilon_s ) \d s +  g( \overline{X}^\varepsilon_T) \bigg], \]
where $( \overline{X}_s^{\varepsilon})_{s \in [t,T]}$ is the optimal process \eqref{eq:opteps} starting at $x$.
For $\tilde\varepsilon \in (0,1]$, recalling the gradient bound \eqref{eq:GradBound}, this implies
\[ \vert \varphi_t^{\varepsilon}(x) - \varphi_t^{\tilde\varepsilon}(x) \vert \leq L^2 (T-t) + \frac{L_f}{2} (T-t) \sup_{t \leq s \leq T} \E [ \vert \overline{X}_s^{\varepsilon} - \overline{X}_s^{\tilde\varepsilon} \vert ] + L_g \E [ \vert \overline{X}_T^{\varepsilon} - \overline{X}_T^{\tilde\varepsilon} \vert ]. \]
However, for every $s \in [t,T]$, 
\[ \E \bigl[ | \overline{X}_s^{\varepsilon} - \overline{X}_s^{\tilde\varepsilon} |  \bigr] = \E \bigg[ \bigg\lvert \int_{t}^s -\nabla \varphi_r^{\varepsilon}(\overline{X}^\varepsilon_r) + \nabla \varphi_r^{\varepsilon}(\overline{X}^{\tilde\varepsilon}_r) \d r + ( \sqrt{\varepsilon} - \sqrt{\tilde\varepsilon} )(B_{s} - B_{t}) \bigg\rvert \bigg], \]
so that $\E \bigl[ | \overline{X}_s^{\varepsilon} - \overline{X}_s^{\tilde\varepsilon} |  \bigr] \leq 2 L (T-t) + ( \sqrt{\varepsilon} + \sqrt{\tilde\varepsilon}) \sqrt{T-t}$. Gathering terms yields
\begin{multline*} 
\vert \varphi_t^{\varepsilon}(x) - \varphi_t^{\tilde\varepsilon}(x) \vert \leq L^2 (T-t) + \frac{L_f}{2} (T-t) [ 2 L ( T -t ) + ( \sqrt{\varepsilon} + \sqrt{\tilde\varepsilon}) \sqrt{T-t} ] \\
+ L_g[ 2 L (T-t) + ( \sqrt{\varepsilon} + \sqrt{\tilde\varepsilon}) \sqrt{T-t} ]. \end{multline*}
Sending $\tilde\varepsilon \rightarrow 0$ concludes, recalling that $L = L_g + T L_f$.
\end{proof}

We are now ready to compute the optimal rate.

\begin{theorem}[Optimal rate] \label{thm:optSC}
If $f$, $g$ are globally Lipschitz and semiconcave, then there exists $C_{\mathrm{opt}} >0$ such that
\begin{equation}
\forall \varepsilon \in (0,1], \quad \sup_{(t,x) \in [0,T] \in \R^d} \vert \varphi^\varepsilon_t (x) - \varphi^0_t (x) \vert \leq - C_{\mathrm{opt}} \, \varepsilon \log \varepsilon.
\end{equation}  
Furthermore, there exists an explicit constant $C(d,T,L_f,L_g,\lambda_f ,\lambda_g) > 0 $ such that
\[ \forall (t,x) \in [0,T] \times \R^d, \quad \varphi^\varepsilon_t (x) - \varphi^0_t (x) \geq \tfrac{d}{2} \varepsilon \log \varepsilon -C(d,T,L_f,L_g,\lambda_f, \lambda_g) \varepsilon, \]
where $\lambda_f$, $\lambda_g$ are the semiconcavity constants of $f$, $g$.
\end{theorem}

\begin{proof}
Subtracting \eqref{eq:HJBeps} to \eqref{eq:convolSub}, a direct computation yields
\begin{equation} \label{eq:KeyLower}
\big[ \partial_s - \nabla \psi^{\varepsilon,\delta}_s \cdot \nabla + \tfrac{\varepsilon}{2} \Delta \big] ( \varphi^{\varepsilon}_s - \varphi^{0,\delta}_s ) \leq - \frac{\varepsilon}{2} \Delta \varphi^{0,\delta}_s + 2 \delta L L_f. 
\end{equation} 
From the regularity given by Lemma \ref{lem:supconv}-\ref{item:supReg}, we can apply Ito's formula along \eqref{eq:OInterpSDE} for functions with generalized derivatives \cite[Chapter 2.10,Theorem 2]{krylov2008controlled} to $\varphi^{\varepsilon} - \varphi^{0,\delta}$, and take expectations. 
Since $\varphi^\varepsilon_T = \varphi^{0}_T = g$, this yields
\begin{equation} \label{eq:StepBefore}
\forall (t,x) \in [0,T] \times \R^d, \quad \varphi^\varepsilon_t (x) - \varphi^{0,\delta}_t (x) \geq - \lVert g - g^\delta \rVert_\infty - 2 \delta L_f L + \frac{\varepsilon}{2} \int_t^T \E \big[ \Delta \varphi^{0,\delta}_s (  Y^{\varepsilon,\delta,x}_s ) \big] \d s. 
\end{equation} 
From Lemma \ref{lem:supconv}-\ref{item:supReg}, $\varphi^{0,\delta}$ is $-1/\delta$-semiconvex, so that for $\tau \in (0,T-t]$,
\begin{equation} \label{eq:LastConvS}
\frac{\varepsilon}{2} \int_t^{t+\tau} \E \big[ \Delta \varphi^{0,\delta}_s (  Y^{\varepsilon,\delta,x}_s ) \big] \d s \geq - \frac{\varepsilon \tau d}{2\delta}. 
\end{equation} 
Since $\mu^{\varepsilon,\delta,x}_s$ is the law of $Y^{\varepsilon,\delta,x}_s$, Proposition \ref{pro:Regularising}-\ref{item:LapLow} and Lemma \ref{lem:supconv}-\ref{item:supAppr} then yield
\[ \varphi^\varepsilon_t (x) - \varphi^{0}_t (x) \geq - 2 L ( 1 + L_f) \delta - \frac{\varepsilon \tau d}{2\delta} + \frac{\varepsilon d}{2} \log ( 2 \pi \varepsilon \tau ) - \frac{\tau L^2}{2} - \frac{\varepsilon d}{2} \log (2 \pi \varepsilon ) - \frac{\varepsilon(T-t-\tau) d \lambda}{2} - 2 \varepsilon d e^{4 L^2 T^2}, \]
for every $(t,x) \in [0,T] \times \R^d$. 
If $T - t \geq \varepsilon$, we can choose $\delta = \tau = \varepsilon$ and conclude.
Otherwise, Lemma \ref{lem:TemrCont} applied to $\eta = \varepsilon$ completes the proof.
\end{proof}

We notice that the stochastic aspects of this proof could be avoided, since Ito's formula is only needed in an averaged sense against $\mu^{\varepsilon,\delta}_s$. 
Seeing $\mu^{\varepsilon,\delta}_s$ as the flow of the adjoint of a well-chosen ``linearised version'' of \eqref{eq:HJBeps} is reminiscent of the nonlinear adjoint method \cite{tran2011adjoint}.

\begin{rem}[Recovering Theorem \ref{thm:subopt}]
Since $\varphi^{0,\delta}_s$ is $-1/\delta$-semiconvex, we could directly lower bound $\Delta \varphi^{0,\delta}_s$ by $-d / \delta$ in \eqref{eq:StepBefore} and set $\delta = \sqrt{\varepsilon}$. 
This would recover the sub-optimal rate from Theorem \ref{thm:subopt}.
\end{rem}

\begin{rem}[Integrated Fisher information]
It is tempting to exploit the sign of the integrated Fisher information that appears in \eqref{ItoLog} after taking expectations.
For instance, we could integrate the Laplacian by parts in the last term of \eqref{eq:StepBefore}.
However, such an approach would yield the sub-optimal rate $\sqrt{\varepsilon}$, because of the $\varepsilon$ factor in front of the $\vert \nabla \log \mu^{\varepsilon,\delta}_s \vert^2$ term in \eqref{ItoLog}.   
\end{rem}

\subsection{Beyond semiconcavity} \label{subsec:NSC}

In this section, we no more assume that $g$ is semiconcave.
Instead, we rely on the ``concavifying'' effect of Hamilton-Jacobi equations \cite{cannarsa2004semiconcave}.
For the sake of simplicity, we take $f \equiv 0$.

\begin{lemma}[Semiconcavity generation] \label{lem:SC+}
For $(\varepsilon,t) \in (0,1] \times [0,T)$, $\varphi^\varepsilon_t$ is $\tfrac{1}{T-t}$-semiconcave.
\end{lemma}

\begin{proof}
A direct computation (Cole-Hopf transform) shows that $-\varepsilon \log \varphi^\varepsilon_t$ solves a time-reversed heat equation. Hence,
\begin{equation} \label{eq:HopfC}
\forall (t,x) \in \R^d, \qquad \varphi^\varepsilon_t (x) = \frac{\varepsilon d}{2} \log ( 2 \pi \varepsilon(T-t) ) - \varepsilon \log \int_{\R^d} \exp \bigg[ - \frac{g(y)}{\varepsilon} - \frac{|y-x|^2}{2 \varepsilon (T-t)} \bigg] \d y, 
\end{equation} 
so that
\[ \nabla^2 \varphi^\varepsilon_t (x) = \frac{1}{T-t} \mathrm{Id} - \frac{\int_{\R^d} (y-x) \otimes (y-x) \exp \bigg[ - \frac{g(y)}{\varepsilon} - \frac{|y-x|^2}{2 \varepsilon (T-t)} \bigg] \d y}{\varepsilon (T-t)^2 \int_{\R^d} \exp \bigg[ - \frac{g(y)}{\varepsilon} - \frac{|y-x|^2}{2 \varepsilon (T-t)} \bigg] \d y}. \]
The second matrix is positive semidefinite, proving the result. 
\end{proof}

This property extends to $\varepsilon = 0$.

\begin{proposition}[Upper bound] \label{pro:UppNSC}
For every $\varepsilon \in (0,1]$,
\begin{equation*} \label{eq:Upper}
\forall (t,x) \in [0,T] \times \R^d, \quad \varphi^{\varepsilon}_t ( x) - \varphi^{0}_t ( x) \leq - \frac{d  \varepsilon}{2} \varepsilon \log \varepsilon + \frac{d \varepsilon}{2} \log T  + \varepsilon (3 L^2_g + 2 L_g).
\end{equation*}
\end{proposition}

\begin{proof}
We first fix $\eta \in (0,T)$ and observe that, because of the $\frac{1}{T-t}$-semi-concavity of $\varphi^0_t$, we easily show that $(t,x) \mapsto \varphi^0_t (x) + \tfrac{\varepsilon}{2} \int_t^{T-\eta} \frac{d}{T-s} \d s$ is a viscosity super-solution of 
    $$ -\partial_tv_t + \frac{1}{2} |\nabla v_t|^2 - \frac{\varepsilon}{2} \Delta v_t = 0, \quad \mbox{ in  } [0,T-\eta] \times \R^d, \quad v_{T-\eta} = \varphi^0_{T-\eta} $$
and we deduce that
\begin{align*} 
\sup_{t \in [0,T-\eta] \times \R^d} \varphi^{\varepsilon}_t(x) - \varphi^0_t(x) &\leq \frac{d\varepsilon}{2} \int_0^{T-\eta} \frac{1}{T-t} \d t + \sup_{x \in \R^d} | \varphi^{\varepsilon}_{T-\eta}(x) - \varphi^0_{T-\eta}(x)|  \\
& \leq \frac{d\varepsilon}{2} \log \frac{T}{\eta} + \sup_{(t,x) \in [T-\eta,T] \times \R^d} | \varphi_t^{\varepsilon}(x) - \varphi^0_t(x)|.  
\end{align*}
Using Lemma \ref{lem:TemrCont}, choosing $\eta = \varepsilon$ concludes.
\end{proof}

We are now ready to conclude the proof of Theorem \ref{thm:opt}.

\begin{proof}[End of the proof of Theorem \ref{thm:opt}]
If $T- t \leq \varepsilon$, the result is given by Lemma \ref{lem:TemrCont} with $\eta = \varepsilon$. 
Hence, we can assume that $0 < \varepsilon < T-t$.
We only perform a small change in the proof of  Theorem \ref{thm:optSC}: from \eqref{eq:KeyLower}, we apply Ito's formula from time $t$ untill $T- \varepsilon$. 
This replaces \eqref{eq:StepBefore} by
\begin{equation} \label{eq:StepNSC}
\varphi^\varepsilon_t (x) - \varphi^{0,\delta}_t (x) \geq \E [ \varphi^\varepsilon_{T-\varepsilon} ( Y^{\varepsilon,\delta,x}_{T-\varepsilon} ) - \varphi^{0,\delta}_{T-\varepsilon} (Y^{\varepsilon,\delta,x}_{T-\varepsilon})] + \frac{\varepsilon}{2} \int_t^{T-\varepsilon} \E \big[ \Delta \varphi^{0,\delta}_s (  Y^{\varepsilon,\delta,x}_s ) \big] \d s. 
\end{equation}
Combining Lemma \ref{lem:supconv}-\ref{item:supReg} and Lemma \ref{lem:TemrCont} applied to $\eta = \varepsilon$,
\[ \E [ \varphi^\varepsilon_{T-\varepsilon} ( Y^{\varepsilon,\delta,x}_{T-\varepsilon} ) - \varphi^{0,\delta}_{T-\varepsilon} (Y^{\varepsilon,\delta,x}_{T-\varepsilon})] \geq - \lVert \varphi^\varepsilon_{T-\varepsilon} - \varphi^0_{T-\varepsilon} \rVert_\infty - \lVert \varphi^0_{T-\varepsilon} - \varphi^{0,\delta}_{T-\varepsilon} \rVert_\infty \geq - \varepsilon (3 L^2_g + 2 L_g) - 2 L_g \delta . \]
For $\tau \in (0,T-t-\varepsilon]$, Lemma \ref{lem:SC+} then gives 
\[ -\frac{1}{2} \int_{t+\tau}^{T-\varepsilon} \int_{\R^d} \Delta \varphi^{\varepsilon}_s \d \mu^{\varepsilon,\delta}_s \d s \geq - \frac{1}{2} \int_{t+\tau}^{T-\varepsilon}  \frac{d}{2 (T-s)} \d s \geq \frac{d}{2} \log \frac{\varepsilon}{T}. \]
This replaces the result of Proposition \ref{pro:Regularising}-\ref{item:LapLow} by
\[ \frac{1}{2} \int_{t+\tau}^{T-\varepsilon} \int_{\R^d} \Delta \varphi^{0,\delta}_s \d \mu^{\varepsilon,\delta, x}_s \d s \geq \frac{d}{2} \log ( 2 \pi \varepsilon \tau ) - \frac{\tau}{2 \varepsilon} L_g^2 - \frac{d}{2} \log (2 \pi \varepsilon ) + \frac{d}{2} \log \frac{\varepsilon}{T} - 2 d e^{4 L^2_g (T-\varepsilon)^2}, \]
which we can plug in \eqref{eq:StepNSC} as previously, together with \eqref{eq:LastConvS}.
Gathering pieces now yields
\[ \varphi^\varepsilon_t (x) - \varphi^{0}_t (x) \geq - 2 L_g \delta - \varepsilon (3 L^2_g + 2 L_g) - \frac{\varepsilon \tau d}{2\delta} + \frac{\varepsilon d}{2} \log ( 2 \pi \varepsilon \tau ) - \frac{\tau L_g^2}{2} - \frac{\varepsilon d}{2} \log (2 \pi \varepsilon ) + \frac{\varepsilon d}{2} \log \frac{\eta}{T} - 2 \varepsilon d e^{4 L^2_g T^2}, \]
for every $(t,x) \in [0,T] \times \R^d$. 
If $T - t - \varepsilon \geq \varepsilon$, we can choose $\delta = \tau = \varepsilon$ and conclude.
Otherwise, Lemma \ref{lem:TemrCont} applied to $\eta = 2 \varepsilon$ completes the proof.
\end{proof}

\section{Explicit solutions and optimality of the rate} \label{sec:Example}

In this section, we provide an example to justify that the rate $O(\varepsilon \log \varepsilon)$ cannot be improved, at least in dimension $d \geq 2$. 
In dimension $d=1$ a sharp example is provided in \cite[Proposition 4.4]{qian2024optimal}.
We rely on the explicit formula \eqref{eq:HopfC} and the Hopf-Lax formula, which corresponds to the $\varepsilon \rightarrow 0$ limit of \eqref{eq:HopfC}. 

\begin{proposition} \label{pro:example}
    For $1 \leq k  \leq d$ and $x= (x_1, \cdots,x_d) \in \R^d$, define the orthogonal  projection $P_k(x) := (x_1,\dots,x_k,0, \dots,0)$ into the first $k$ coordinates. Let $\varphi^{k,\varepsilon}$ and $\varphi^{k,0}$ be respectively the solutions to \eqref{eq:HJBeps} and \eqref{eq:HJB0} with terminal data $g_k = - |P_k(x)|$ and source term $f \equiv 0$. Then, for all $T >0$ and all $t \in [0,T)$ we have the expansion
     $$ \varphi^{k,\varepsilon}_t(0) = \varphi^{k,0}_t(0) + \frac{k-1}{2} \varepsilon \log \varepsilon  - \frac{k-1}{2} \varepsilon \log (T-t) - \varepsilon \log \frac{k \sqrt{\pi}}{ 2^{\frac{k-1}{2}} \Gamma(\frac{k}{2}+1)  } + o(\varepsilon).$$
\end{proposition}

\begin{proof}
Using the Hopf-Lax formula, for all $(\tau,x) \in (0,T] \times \R^d$ we have
\[ \varphi_{T-\tau}^{k,0}(x) = \inf_{y \in \R^d} \bigl \{ -|P_k(y)| + \frac{1}{2\tau} |x-y|^2 \bigr \} = -|P_k(x)| - \frac{\tau}{2}.
\]
Moreover, the infimum is obtained for $y =  x + \tau \frac{P_k(x)}{\vert P_k (x) \vert}$ if $ P_k(x) \neq 0$ and for any $y = x + z$ with $z \in \mathrm{Span}(e_1,\dots,e_k)$ and $\vert z \vert = \tau$
such that $|P_k(y)| = \tau$ when $ P_k (x) = 0$. On the other hand, using \eqref{eq:HopfC} for all $(\tau,x) \in (0,T] \times \R^d$, 
\[ \varphi_{T-\tau}^{k, \varepsilon}(x) =  - \varepsilon \log \Bigl[ \int_{\R^d} e^{\frac{-1}{\varepsilon} \bigl( -|P_k(z+x)|+ \frac{1}{2 \tau} | z|^2 \bigr)} \frac{\d z}{(2 \pi \varepsilon \tau)^{d/2}} \Bigr] \]
Now we decompose $|z|^2 = |P_k(z)|^2 + |z-P_k(z)|^2$ and introduce the two variables $z' = (z_1,\dots,z_k)$ and $z''=(z_{k+1},\dots,z_d)$ to compute, with the slight abuse of notation $P_k(x) = (x_1, \dots,x_k) \in \R^k$,
\begin{align*}
 \varphi_{T-\tau}^{k, \varepsilon}(x) &= - \varepsilon \log \Bigl[ \int_{\R^d} e^{\frac{-1}{\varepsilon} \bigl( -|P_k(z+x)|+ \frac{1}{2 \tau} | P_k(z)|^2 \bigr)} e^{-\frac{1}{2\varepsilon\tau} |z-P_k(z)|^2} \frac{\d z}{(2 \pi \varepsilon \tau)^{d/2}} \Bigr] \\
    &= - \varepsilon \log \Bigl[ \int_{\R^k}  e^{\frac{-1}{\varepsilon} \bigl( -|P_k(x) +z'|+ \frac{1}{2 \tau} | z'|^2 \bigr)} \frac{\d z'}{(2 \pi \varepsilon \tau)^{k/2}} \int_{\R^{d-k}} e^{-\frac{1}{2\varepsilon\tau} |z''|^2} \frac{\d z''}{(2 \pi \varepsilon \tau)^{(d-k)/2}}\Bigr] \\
    &=- \varepsilon \log \Bigl[ \int_{\R^k}  e^{\frac{-1}{\varepsilon} \bigl( -|z'|+ \frac{1}{2 \tau} | z' - P_k(x)|^2 \bigr)} \frac{\d z'}{(2 \pi \varepsilon \tau)^{k/2}} \Bigr].
\end{align*}

For any $\tau \in (0,T]$, we expect that the rate deteriorates when $P_k (x) = 0$, since this is precisely where the map $ y \mapsto -|P_k(y)| + \frac{1}{2\tau}|x-y|^2$ has several minima. From now on we take $x=0$. We have $ \varphi^0_{T-\tau}(0) = -\tau/2$, while rewriting $-|y| + \frac{1}{2\tau}|y|^2 = \frac{1}{2\tau}(|y|-\tau)^2 - \frac{\tau}{2}$,
we get
$$ \varphi_{T-\tau}^{k,\varepsilon} (0) = - \frac{\tau}{2} + \frac{k}{2} \varepsilon \log ( 2 \pi \varepsilon\tau) -    \varepsilon \log \int_{\R^k} e^{-\frac{1}{\varepsilon} \frac{1}{2 \tau} (|y|- \tau)^2 } \d y . $$
By a change of variables for radial functions and the change of variable $s = \tfrac{r-\tau}{\sqrt{\varepsilon\tau}}$, we get
$$  \int_{\R^{k}}  e^{-\frac{1}{\varepsilon} \frac{1}{2 \tau} (|y|- \tau)^2 } \d y = C_k \int_0^{\infty} e^{ - \frac{1}{\varepsilon} \frac{1}{2 \tau} (r-\tau)^2 } r^{k-1} \d r = C_k \sqrt{\varepsilon \tau} \int_{ - \frac{\sqrt{\tau}}{\sqrt{\varepsilon}}}^{+\infty} e^{-s^2/2}( \sqrt{\varepsilon \tau} s + \tau)^{k-1} \d s, $$  
where $C_k = k \pi^{k/2} / \Gamma(k/2+1)$ is $k$ times the Lebesgue measure of the unit ball of $\R^k$. We then easily see that 
$$ \lim_{\varepsilon \rightarrow 0^+} \int_{ - \frac{\sqrt{\tau}}{\sqrt{\varepsilon}}}^{+\infty} e^{-s^2/2}( \sqrt{\varepsilon \tau} s + \tau)^{k-1} \d s = \sqrt{2\pi} \tau^{k-1}, $$
from which the result follows.
\end{proof}

\section*{Acknowledgements}
\addcontentsline{toc}{section}{Acknowledgement}

The first author is grateful to Professor Daniel Lacker for suggesting to him this elegant proof of Proposition \ref{pro:Regularising}-\ref{item:RegEnt}.

\printbibliography
\addcontentsline{toc}{section}{References}

\end{document}